\documentclass[11pt]{amsart}

\usepackage{amsmath,amssymb,amsfonts,amsthm}
\usepackage{mathrsfs}
\usepackage{hyperref}
\usepackage{url}       

\usepackage[all]{xy}

\newtheorem{Thm}{Theorem}[section]

\newtheorem{Lem}[Thm]{Lemma}

\newtheorem{Conj}[Thm]{Conjecture}
\newtheorem{Def}[Thm]{Definition}
\newtheorem{Rem}[Thm]{Remark}
\newtheorem{Question}[Thm]{Question}

\newcommand{\Fpbar}{\overline{\mathbb F}_p}

\newcommand{\GL}{\operatorname{GL}}
\newcommand{\Spec}{\operatorname{Spec}}
\newcommand{\Div}{\operatorname{Div}}
\newcommand{\Gal}{\operatorname{Gal}}
\newcommand{\Ker}{\operatorname{Ker}}

\title[Finiteness of Geometric Representations for Varieties]{On the Finiteness of Geometric Representations for Varieties over Finite Fields}

\author{Yufan Luo}

\address{Shanghai Institute for Mathematics and Interdisciplinary Sciences (SIMIS), Shanghai 200433, China
}
\address{Research Institute of Intelligent Complex Systems, Fudan University, Shanghai
	200433, China}
\email{yufanluo@hotmail.com}



\begin{document}
	
\begin{abstract}
	Let $p$ be a prime number, and let $k$ be a finite field of characteristic different from $p$. Let $X$ be a normal geometrically connected variety over $k$, let $\overline X$ be a compactification of $X$, and let $Z=\overline X\setminus X$. Let $D$ be an effective Cartier divisor on $\overline X$ whose support is contained in $Z$. Motivated by Hiranouchi's Hermite--Minkowski type theorem for varieties over finite fields, we formulate a finiteness conjecture for continuous semisimple geometric representations
	$$
	\pi_1(X,D)\longrightarrow \operatorname{GL}_n(F),
	$$
	where $\pi_1(X,D)$ is Hiranouchi's fundamental group with ramification bounded by $D$, and $F$ is an algebraically closed field of characteristic $p$ endowed with the discrete topology. We prove this conjecture for odd $p$ in the following two cases: for curves with arbitrary ramification bound $D$, and for varieties of arbitrary dimension in the tame case, namely $D=0$. Furthermore, for arbitrary $p$, we prove the finiteness for those representations which admit a lift to characteristic zero.
\end{abstract}

	\maketitle
	\tableofcontents

	\section{Introduction}
	
	The classical Hermite--Minkowski theorem asserts that, for a given number field, there exist only finitely many extensions of a prescribed degree which are unramified outside a fixed finite set of primes. This fundamental finiteness theorem was generalized to higher-dimensional arithmetic schemes by Faltings. In the function field setting, Deligne introduced the notion of ramification bounded by an effective Cartier divisor $D$ at infinity for lisse $\ell$-adic sheaves, as explained by Esnault and Kerz
	\cite[Def.~4.6]{EK11} and \cite[Def.~3.3]{MR3058662}. In this framework, Hiranouchi introduced a fundamental group $\pi_1(X,D)$ classifying finite étale covers of $X$ whose ramification is bounded by $D$, and established a corresponding Hermite--Minkowski type theorem for varieties over finite fields.
	
	Let $X$ be a variety over a finite field $k$, i.e. a connected separated
	scheme of finite type over $k$. Let $\overline X$ be a compactification of $X$, and let $Z=\overline X\setminus X$. Let $\Div_Z^+(\overline X)$ denote the monoid of effective Cartier divisors
	$D$ on $\overline X$ whose support is contained in $Z$. For $
	D\in \Div_Z^+(\overline X)$,
	Hiranouchi \cite[Definition~2.4]{MR3622140} defined a quotient
	$$
	\pi_1(X\subset \overline X,D)
	$$
	of the étale fundamental group $\pi_1(X)$, which classifies finite étale
	covers of $X$ whose ramification is bounded by $D$. We omit base points from
	the notation throughout. In the following, we write $\pi_1(X,D)$ for
	$\pi_1(X\subset \overline X,D)$ when we need not specify the fixed compactification $\overline X$ of $X$. We recall that Hiranouchi's definition is a covering-theoretic analogue of the bounded ramification condition for lisse $\ell$-adic sheaves, rather than a definition in terms of conductors of $\ell$-adic sheaves. It is first given for finite extensions of complete discrete valuation fields, then for finite étale covers of curves, and finally for higher-dimensional varieties by testing after pullback to curves. It is proved in \cite[Theorem~3.4]{MR3622140} that $\pi_1(X,D)$ is
	small; that is, for each positive integer $m$, there are only finitely many
	finite étale covers of $X$ of degree at most $m$ whose ramification is
	bounded by $D$.
	
	A notable application of this finiteness property concerns continuous representations of $\pi_1(X,D)$. Let $F$ be an algebraically closed field endowed with the discrete topology. A continuous representation
	$$
	\rho:\pi_1(X,D)\longrightarrow \GL_n(F)
	$$
	is called geometric, in the sense of \cite[Definition 3.5]{MR3622140}, if its image coincides with the image of the geometric fundamental group. More precisely, if
	$$
	\pi_1(X,D)^0:=\Ker\bigl(\pi_1(X,D)\longrightarrow \pi_1(\Spec k)\bigr),
	$$
	then $\rho$ is geometric if
	$$
	\rho(\pi_1(X,D))=\rho(\pi_1(X,D)^0).
	$$
	When $X$ is normal, this condition is equivalent to saying that the finite extension of the function field $k(X)$ cut out by $\rho$ contains no non-trivial constant field extension.
	
	In \cite[Corollary 3.6]{MR3622140}, it was proved that, for a normal variety $X$, there are only finitely many isomorphism classes of continuous semisimple geometric representations
	$$
	\pi_1(X,D)\longrightarrow \GL_n(F)
	$$
   in the following two cases: either the image is solvable, or the coefficient field $F$ has characteristic zero.
	
	The case where $F$ has positive characteristic and the image is not assumed to be solvable remains open in general. This case is substantially different from the characteristic zero case. In characteristic zero, Jordan's theorem asserts that every finite subgroup of $\GL_n(F)$ contains an abelian normal subgroup of index bounded only in terms of $n$. This kind of uniform group-theoretic control is no longer available in characteristic $p>0$. Indeed, already for $n\geq 2$, the finite groups $\GL_n(\mathbb F_{p^r})$ embed into $\GL_n(F)$ for all $r\geq 1$, and they do not contain abelian normal subgroups of uniformly bounded index. Thus, Hiranouchi's finiteness theorem for representations with solvable image does not imply the desired finiteness for arbitrary semisimple representations in positive characteristic.
	
	Motivated by Hiranouchi's result and \cite[Theorem 1.3]{afinite}, we formulate the following conjecture.
	
	\begin{Conj}\label{conj:main}
		Let $p$ be a prime number, and let $k$ be a finite field of characteristic different from $p$. Let $X$ be a normal geometrically connected variety over $k$, let $\overline X$ be a compactification of $X$, and put
		$
		Z=\overline X\setminus X.
		$
		Let
		$
		D\in \Div_Z^+(\overline X)
		$
		be an effective Cartier divisor. Then, for every positive integer $n$, there are only finitely many isomorphism classes of continuous semisimple geometric representations
		$$
		\rho:\pi_1(X,D)\longrightarrow \GL_n(F),
		$$
		where $F$ is an algebraically closed field of characteristic $p$ endowed with the discrete topology. 
	\end{Conj}
\begin{Rem}
	If the coefficient field is replaced by a fixed finite field $F_0$ of
	characteristic $p$, then the corresponding finiteness statement follows
	immediately from the smallness of $\pi_1(X,D)$, and no semisimplicity or
	geometricity assumption is needed. The point of Conjecture~\ref{conj:main}
	is that the coefficient field $F$ is algebraically closed, hence infinite.
	In this setting, smallness alone gives no uniform control on the finite
	field of definition of a representation. Without the semisimplicity and
	geometricity assumptions, one can easily construct counterexamples already
	in the abelian case.
\end{Rem}
	
	The purpose of this paper is to verify Conjecture \ref{conj:main} in three important cases. The first case is the one-dimensional case, where $X$ is a curve. In this case, we prove the conjecture for an arbitrary ramification bound $D$. 
	
	\begin{Thm}\label{main1}
		Let $p$ be an odd prime number, and let $k$ be a finite field of characteristic different from $p$. Let $X$ be a normal geometrically connected curve over $k$, let $\overline X$ be a compactification of $X$, and put
		$
		Z=\overline X\setminus X.
		$
		Let
		$
		D\in \operatorname{Div}_Z^+(\overline X).
		$
		Then, for every positive integer $n$, there are only finitely many isomorphism classes of continuous semisimple geometric representations
		$$
		\rho:\pi_1(X,D)\longrightarrow \operatorname{GL}_n(F),
		$$
		where $F$ is an algebraically closed field of characteristic $p$ endowed with the discrete topology.
	\end{Thm}

The second case is the tame case, corresponding to $D=0$ for varieties with arbitrary dimension.
	
	\begin{Thm}\label{main2}
		Let $p$ be an odd prime number, and let $k$ be a finite field of characteristic different from $p$. Let $X$ be a normal geometrically connected variety over $k$, and let $\overline X$ be a compactification of $X$. Then, for every positive integer $n$, there are only finitely many isomorphism classes of continuous semisimple geometric representations
		$$
		\rho:\pi_1(X,0)\longrightarrow \GL_n(F),
		$$
		where $F$ is an algebraically closed field of characteristic $p$ endowed with the discrete topology.
	\end{Thm}
		\begin{Rem}
			\begin{enumerate}
				\item When $D=0$, $\pi_1(X,0)$ is naturally identified with the tame fundamental group of $X$ with respect to the chosen compactification. In what follows we also write
				$
				\pi_1^t(X)
				$
				for this group.
			\item Although $\pi_1^t(X)$ is known to be topologically finitely generated, this only gives the corresponding finiteness statement when the coefficient field is replaced by a fixed finite field. The point of Theorem~\ref{main2} is that the coefficient field $F$ is algebraically closed, hence infinite. Thus finite generation alone gives no uniform control on the finite field of definition of the representations.
			\end{enumerate}
	\end{Rem}

The third case handles representations that admit a lift to characteristic zero.

	\begin{Thm}\label{main3}
	Let $p$ be a prime number, and let $k$ be a finite field of characteristic different from $p$. Let $X$ be a normal geometrically connected variety over $k$, let $\overline X$ be a compactification of $X$, and put
	$
	Z=\overline X\setminus X.
	$
	Let $D\in \Div_Z^+(\overline X)$. Suppose that $F$ is an algebraically closed field of characteristic $p$ which is endowed with the discrete topology. Then, for every positive integer $n$, there are only finitely many isomorphism classes of continuous semisimple geometric representations
	$$
	\rho:\pi_1(X,D)\longrightarrow \GL_n(F),
	$$
	that admit a lift to characteristic zero.
	\end{Thm}
	
		Let us briefly explain the idea of the proof. The curve case follows from a finiteness theorem \cite[Theorem 1.3]{afinite} for continuous semisimple geometric mod $p$ Galois representations over global function fields with bounded Artin conductor. For the tame case in higher dimension, we use a tame Lefschetz theorem proved by Drinfeld, Esnault and Kindler. In the form needed here, it implies that for a suitable smooth curve $C$ mapping to $X$, the induced homomorphism on tame fundamental groups
	$$
	\pi_1^t(C)\longrightarrow \pi_1^t(X)
	$$
	is surjective. Therefore it reduces to the curve case. Moreover, the proof of Theorem \ref{main3} combines Deligne's finiteness theorem
	with Hiranouchi's finiteness result for the abelian geometric quotient.
	
	The paper is organized as follows. In Section 2, we prove the curve case. In Section 3, we recall the tame Lefschetz theorem and prove the tame case in arbitrary dimension. Section 4 is devoted to the proof of Theorem \ref{main3}.
	
	 Throughout this paper, $p$ is a prime number, $\mathbb{F}_{p}$ denotes the finite field of order $p$, and $\mathbb{Q}_{p}$ is the field of $p$-adic numbers. We fix algebraic closures $\overline{\mathbb{F}}_{p}$ and $\overline{\mathbb{Q}}_{p}$ of $\mathbb{F}_{p}$ and $\mathbb{Q}_{p}$, respectively.
	 
	\section{The curve case}
	
	In this section we prove Conjecture \ref{conj:main} when $X$ is a curve for an odd prime $p$.
	
	Let $K$ be a global function field of characteristic different from $p$, and let $G_K$ be its absolute Galois group. We use the following finiteness theorem for mod $p$ representations over global function fields.
	
	\begin{Thm}\label{thm:function-field}
		Let $p$ be an odd prime number. Let $K$ be a global function field of characteristic different from $p$, and let $R$ be an effective divisor of $K$. Then, for every positive integer $n$, there are only finitely many isomorphism classes of continuous semisimple representations
		$$
		\rho:G_K\longrightarrow \GL_n(\Fpbar)
		$$
		such that:
		\begin{enumerate}
			\item the Artin conductor of $\rho$ is bounded by $R$;
			\item $\rho$ is geometric, that is, the fixed field of $\Ker(\rho)$ contains no non-trivial constant field extension of $K$.
		\end{enumerate}
	\end{Thm}
	\begin{proof}
		It is \cite[Theorem 1.3]{afinite}.
	\end{proof}

	We now translate bounded ramification for curves into bounded Artin conductor.
	
	\begin{Lem}\label{lem:bounded-to-artin}
		Let $X$ be a smooth geometrically connected curve over $k$, let $\overline X$ be a smooth compactification of $X$, and let $Z=\overline X\setminus X$. Let 
		$$
		D=\sum_{z\in Z} m_z[z]\in \Div_Z^+(\overline X).
		$$
		Fix a positive integer $n$. There exists an effective divisor $R_D$ on $\overline X$, depending only on $D$ and $n$, such that for every continuous representation
		$$
		\rho:\pi_1(X,D)\longrightarrow \GL_n(\Fpbar),
		$$
		the corresponding representation of the absolute Galois group $G_{k(X)}$ of the function field of $X$ has Artin conductor bounded by $R_D$.
	\end{Lem}
	
	\begin{proof}
		Let $\rho$ be such a representation. Since $\GL_n(\Fpbar)$ is discrete and $\pi_1(X,D)$ is profinite, the image of $\rho$ is finite. Hence the local ramification groups act through finite quotients.
		
		For every point $z\in Z$, by \cite[Def. 2.2]{MR3622140}, the condition that $\rho$ factors through $\pi_1(X,D)$ means that the upper ramification subgroup $G_{k(X)_z}^{m_z+}$ acts trivially. Since the dimension of $\rho$ is equal to $n$, the Swan conductor at $z$ is bounded by a constant depending only on $m_z$ and $n$, cf. \cite[Lem. 2.18]{MR3622140}. Moreover the tame part contributes at most $n$ to the Artin conductor. Therefore the local Artin conductor at $z$ is bounded by a constant depending only on $m_z$ and $n$.
		
		At points of $X$, the representation is unramified. Thus the Artin conductor is bounded by an effective divisor supported on $Z$, depending only on $D$ and $n$. This proves the assertion.
	\end{proof}
	
	Now we prove Theorem \ref{main1} as follows:

	\begin{proof}[Proof of Theorem \ref{main1}]
		Since $\pi_1(X,D)$ is profinite and $\GL_n(F)$ is discrete, every continuous representation has finite image. For a finite group $G$, every semisimple representation of $G$ over an algebraically closed field of characteristic $p$ is defined over $\overline{\mathbb F}_p$. Therefore, after replacing a representation by an isomorphic one, we may assume that its image lies in $\GL_n(\overline{\mathbb F}_p)$. Thus it suffices to prove the theorem for $F=\overline{\mathbb F}_p$. 
		
	As in the proof of \cite[Theorem 3.4]{MR3622140}, after replacing the compactification $\overline{X}$ with its normalization, and invoking the corresponding surjection on fundamental groups with bounded ramification, we may assume that $\overline X$ is a normal curve. Since $k$ is perfect, every normal curve over $k$ is regular, hence smooth over $k$. In particular, both $X$ and $\overline{X}$ are smooth. Let
		$
		K=k(X).
		$
		Every continuous representation
		$$
		\rho:\pi_1(X,D)\longrightarrow \GL_n(\Fpbar)
		$$
		may be viewed as a representation of $G_K$ which is unramified on $X$ and whose ramification at the points of $Z$ is bounded by $D$. By Lemma \ref{lem:bounded-to-artin}, its Artin conductor is bounded by an effective divisor $R_D$ depending only on $D$ and $n$.
		
		Moreover, the geometric condition for $\rho$ as a representation of $\pi_1(X,D)$ is equivalent to the condition that the fixed field of $\Ker(\rho)$ contains no non-trivial constant field extension of $K$. Therefore Theorem \ref{thm:function-field} applies and gives the desired finiteness.
	\end{proof}
	
	\section{The tame case in arbitrary dimension}
	
	In this section we prove Conjecture \ref{conj:main} in the tame case $D=0$ for an odd prime $p$.
    \subsection{Proof of Theorem \ref{main2}}
	We first recall the Lefschetz theorem for tame coverings that we need. The following formulation is a consequence of Drinfeld's theorem and its extension by Esnault and Kindler.

\begin{Thm}\label{thm:tame-lefschetz}
	Let $X$ be a smooth geometrically connected variety
	over a finite field $k$, and let
	$
	X\hookrightarrow \overline X
	$
	be a projective normal geometrically connected compactification. Then there
	exists a smooth geometrically connected curve $C$ over $k$ and a morphism
	$$
	C\longrightarrow X
	$$
	such that the induced homomorphism on tame fundamental groups
	$$
	\pi_1^t(C)\longrightarrow \pi_1^t(X)
	$$
	is surjective.
\end{Thm}

\begin{proof}
	If $\dim X=1$, then we can  take $C=X$. Thus we may assume that $\dim X\geq 2$. Let
	$
	Z:=\overline X\setminus X
	$
	with its reduced closed subscheme structure. Write
	$$
	Z^{(1)}=\bigcup_i Z_i
	$$
	for the union of the irreducible components of $Z$ of codimension one in
	$\overline X$, and let $Z^{(\geq 2)}$ be the union of the remaining irreducible
	components. Define
	$$
	\Sigma
	=
	\operatorname{Sing}(\overline X)
	\cup Z^{(\geq 2)}
	\cup \bigcup_i \operatorname{Sing}(Z_i)
	\cup \bigcup_{i\neq j}(Z_i\cap Z_j).
	$$
	Here each $Z_i$ is endowed with its reduced induced scheme structure, and $\operatorname{Sing}(Y)$ denotes the singular locus of a scheme $Y$, endowed with the reduced closed subscheme structure. Since
	$\overline X$ is normal, its singular locus $\operatorname{Sing}(\overline X)$ has codimension at
	least two in $\overline X$. Moreover, each $\operatorname{Sing}(Z_{i})$ and the intersections $Z_i\cap Z_j$ have codimension at least
	two in $\overline X$. Hence $\Sigma$ has codimension at least two in
	$\overline X$. By construction, $\overline X\setminus\Sigma$ is smooth. Moreover,
	$$
	(\overline X\setminus\Sigma)\cap(\overline X\setminus X)
	=Z\setminus \Sigma=
	\bigsqcup_i (Z_i\setminus\Sigma)
	$$
	is a disjoint union of smooth schemes, and hence is smooth.

	By Drinfeld's theorem \cite[Proposition C.1]{MR3024821} and \cite[Section 6]{MR3556253}, there exists a smooth projective curve
	$$
	\overline C\subset \overline X\setminus\Sigma,
	$$
	 such that
	$$
	C:=\overline C\cap X
	$$
   is geometrically connected and the induced homomorphism on tame fundamental groups
	$$
	\pi_1^t(C)\longrightarrow \pi_1^t(X)
	$$
	is surjective.
\end{proof}

	\begin{Lem}\label{lem:restriction-geometric}
		Let
		$$
		f:C\longrightarrow X
		$$
		be a morphism of geometrically connected varieties over $k$ such that the induced homomorphism
		$$
		f_*:\pi_1(C,0)\longrightarrow \pi_1(X,0)
		$$
		is surjective. If
		$$
		\rho:\pi_1(X,0)\longrightarrow \GL_n(F)
		$$
		is a continuous geometric representation, then
		$$
		\rho\circ f_*:\pi_1(C,0)\longrightarrow \GL_n(F)
		$$
		is also geometric.
	\end{Lem}
	
	\begin{proof}
		We have a commutative diagram
		$$
		\xymatrix{
			\pi_1(C,0) \ar[r]^{f_*} \ar[d] & \pi_1(X,0) \ar[d] \\
			\pi_1(\Spec k) \ar@{=}[r] & \pi_1(\Spec k).
		}
		$$
	
		We claim that
		$$
		f_*(\pi_1(C,0)^0)=\pi_1(X,0)^0.
		$$
		Indeed, the inclusion
		$$
		f_*(\pi_1(C,0)^0)\subseteq \pi_1(X,0)^0
		$$
		follows from the commutativity of the diagram. Conversely, let
		\(g\in \pi_1(X,0)^0\). Since \(f_*\) is surjective, there exists
		\(h\in \pi_1(C,0)\) such that \(f_*(h)=g\). By commutativity, the image of
		\(h\) in \(\pi_1(\Spec k)\) is equal to the image of \(g\), which is trivial.
		Thus \(h\in \pi_1(C,0)^0\), and hence
		\(g\in f_*(\pi_1(C,0)^0)\). This proves the claim.
		
		Therefore
		$$
		(\rho\circ f_*)(\pi_1(C,0)^0)
		=
		\rho(\pi_1(X,0)^0)
		=
		\rho(\pi_1(X,0))
		=
		(\rho\circ f_*)(\pi_1(C,0)),
		$$
		where the middle equality follows from the geometricity of \(\rho\). Hence
		\(\rho\circ f_*\) is geometric.
	\end{proof}
	
Finally, we prove Theorem \ref{main2} as follows:
	\begin{proof}[Proof of Theorem \ref{main2}]
		For any non-empty open subscheme $U\hookrightarrow X$, the homomorphism $\pi_1(U,0)\to \pi_{1}(X,0)$ is surjective using the same argument in the proof of \cite[Theorem 3.4]{MR3622140}. Thus we may assume that $X$ is quasi-projective. 
		
		Following the proof of \cite[Theorem 3.4]{MR3622140}, by replacing $X$ with its smooth locus, replacing the compactification with its normalization, and invoking the corresponding surjection on fundamental groups with bounded ramification, we may assume that $X$ is smooth and $\overline X$ is normal and projective. By the tame Lefschetz theorem \ref{thm:tame-lefschetz}, there exists a smooth geometrically connected curve $C$ over $k$ and a morphism
		$$
		f:C\longrightarrow X
		$$
		such that the induced homomorphism
		$$
		f_*:\pi_1(C,0)\longrightarrow \pi_1(X,0)
		$$
		is surjective.
		
		Let
		$$
		\rho:\pi_1(X,0)\longrightarrow \GL_n(F)
		$$
		be a continuous semisimple geometric representation. Then the pullback
		$$
		\rho_C:=\rho\circ f_*:\pi_1(C,0)\longrightarrow \GL_n(F)
		$$
		is continuous, semisimple, and geometric by Lemma \ref{lem:restriction-geometric}.
		
		By Theorem \ref{main1}, applied to the curve $C$ with $D=0$, there are only finitely many isomorphism classes of possible representations $\rho_C$. Since $f_*$ is surjective, the representation $\rho$ is uniquely determined by $\rho_C$. Hence there are only finitely many isomorphism classes of such representations $\rho$.
	\end{proof}
	
	\begin{Rem}
		The results above leave open the full bounded-ramification case in higher dimensions. The main obstacle is that, unlike the tame case, there is currently no Lefschetz theorem for Hiranouchi's fundamental group $\pi_1(X,D)$ with arbitrary $D$ that is strong enough to reduce the problem to curves. If such a Lefschetz theorem were available, the curve case proved in Section 2 would imply Conjecture \ref{conj:main} in full generality.
	\end{Rem}
	
  \subsection{Further remarks}
  
  It is natural to ask whether Conjecture~\ref{conj:main} can be reduced to the tame case. For a general effective divisor $D$, however, the quotient $\pi_1(X,D)$ still allows wild ramification, although bounded by $D$. Such a
  reduction would require a much stronger uniform statement. Namely, one would
  need a finite étale cover
  $$
  Y\longrightarrow X
  $$
  depending only on $X,\overline X,D$ and $n$, such that for every continuous
  semisimple geometric representation
  $$
  \rho:\pi_1(X,D)\longrightarrow \GL_n(F),
  $$
  the pullback representation $\rho|_{\pi_1(Y)}$ is tame. In view of Hiranouchi's smallness theorem \cite[Theorem 3.4]{MR3622140}, it would be enough to prove the
  following weaker uniform statement.
  
  \begin{Conj}\label{conj:bounded-individual}
  	Fix $X,\overline X,D$ and $n$. There exist a constant $M$ and an effective
  	Cartier divisor $D'$ supported on $\overline X\setminus X$, depending only on
  	$X,\overline X,D$ and $n$, such that for every continuous semisimple geometric
  	representation
  	$$
  	\rho:\pi_1(X,D)\longrightarrow \GL_n(\overline{\mathbb F}_p),
  	$$
  	there exists a finite étale cover
  	$$
  	Y_\rho\longrightarrow X
  	$$
  	satisfying the following properties:
  	\begin{enumerate}
  		\item $\deg(Y_\rho/X)\le M$;
  		\item the ramification of $Y_\rho\to X$ is bounded by $D'$;
  		\item the pullback representation
  		$$
  		\rho|_{\pi_1(Y_\rho)}
  		$$
  		is tame.
  	\end{enumerate}
  \end{Conj}
  
  Indeed, by Hiranouchi's smallness theorem, there are only finitely many
  possibilities for the covers $Y_\rho\to X$ satisfying the first two conditions.
  Taking a common finite étale cover dominating all of them would then give a
  single cover $Y\to X$, depending only on $X,\overline X,D$ and $n$, which
  makes every such representation tame after pullback.
	
	\section{Representations liftable to characteristic zero}
	
	In this section we prove a finiteness result for mod $p$ representations which admit a lift to characteristic zero.
	
	\subsection{Proof of Theorem \ref{main3}} We first make precise what we mean by such a lift.

   \begin{Def}
   	Let $F$ be an algebraically closed field of characteristic $p$ equipped
   	with the discrete topology. We say that a continuous representation
   	\[
   	\rho:\pi_1(X,D)\longrightarrow \GL_n(F)
   	\]
   	admits a lift to characteristic zero if there exist a finite extension
   	$E/\mathbb Q_p$, with ring of integers $\mathcal O_E$ and residue field
   	$\kappa_E$, an embedding $\kappa_E\hookrightarrow F$, and a continuous
   	representation
   	\[
   	\widetilde\rho:\pi_1(X,D)\longrightarrow \GL_n(\mathcal O_E)
   	\]
   	such that, after extension of scalars through $\kappa_E\hookrightarrow F$,
   	the reduction of $\widetilde\rho$ modulo the maximal ideal of $\mathcal O_E$ is isomorphic to $\rho$. Here the group $\GL_{n}(\mathcal{O}_{E})$ is equipped with the $p$-adic topology induced from the topological ring $\mathcal{O}_{E}$.
   \end{Def}
	
     \begin{Thm}[Deligne's finiteness theorem]\label{Deligne'sfinitenesstheorem}
     Let $k$ be a finite field of characteristic different from $p$. Let $X$ be a geometrically connected normal scheme of finite type defined over $k$, $X\hookrightarrow \overline{X}$ be a normal compactification, $D$ be an effective Cartier divisor with support in $\overline{X}  \setminus X$, $n$ be a natural number. Then there are only finitely many isomorphism classes of irreducible
     	continuous representations
     	\[
     	\rho:\pi_1(X,D)\to \mathrm{GL}_n(\overline{\mathbb Q}_{p})
     	\]
     up to twist by a character of $\pi_1(\operatorname{Spec}k)$.
     \end{Thm}
	\begin{proof}
	It follows from \cite[Theorem 3.1]{MR3687121} and \cite[Lemma 2.17]{MR3622140}. See also \cite[Theorem 1.1]{MR3058662}.
	\end{proof}
	
	We now prove Theorem \ref{main3} as follows.

	\begin{proof}[Proof of Theorem \ref{main3}]
		As in the proof of Theorem \ref{main1}, we may assume that $F=\overline{\mathbb F}_p$. Furthermore, as in the proof of \cite[Theorem 3.4]{MR3622140}, replacing the compactification
		by its normalization and using the corresponding surjection on fundamental groups
		with bounded ramification, we may also assume that $\overline X$ is normal.
		
	Put $$ G=\pi_1(X,D),\qquad G^0=\ker\bigl(G\to \pi_1(\operatorname{Spec}k)\bigr). $$ We shall use the following consequence of \cite[Lemma 3.7]{MR3622140}. Let $$ G^{\mathrm{ab},0} := \ker\bigl(G^{\mathrm{ab}}\to \pi_1(\operatorname{Spec}k)^{\mathrm{ab}}\bigr). $$ Then $G^{\mathrm{ab},0}$ is finite. It follows that there are only finitely many continuous geometric characters $ G\longrightarrow \overline{\mathbb F}_p^\times $.
	 Let $$ \rho:G\longrightarrow \mathrm{GL}_n(\overline{\mathbb F}_p) $$ be a continuous semisimple geometric representation admitting a lift to characteristic zero. Choose such a lift $$ \widetilde\rho:G\longrightarrow \mathrm{GL}_n(\mathcal O_E). $$ 
	 
	 Let $V$ denote the semisimplification of
	 $\widetilde{\rho}~ \otimes_{\mathcal O_E}\overline{\mathbb Q}_p$. Then $V$ is a semisimple $\overline{\mathbb Q}_p$-representation of $G$. Moreover, by the Brauer--Nesbitt theorem, the semisimplified reduction modulo $p$ of any $G$-stable lattice in $V$ is isomorphic to $\rho$, since $\rho$ is already semisimple. Write $$ V=\bigoplus_i V_i $$ as a direct sum of irreducible $\overline{\mathbb Q}_p$-representations. By Theorem \ref{Deligne'sfinitenesstheorem}, there exists a finite set $\mathcal S$ of irreducible $\overline{\mathbb Q}_p$-representations of $G$, of dimensions at most $n$, such that every $V_i$ is of the form $$ W\otimes \chi $$ with $W\in\mathcal S$ and with $\chi$ a continuous character of $\pi_1(\operatorname{Spec}k)$. For each $W\in\mathcal S$, choose a $G$-stable lattice and let $ \overline W $ denote the semisimplified reduction modulo $p$. Since $\mathcal S$ is finite, only finitely many irreducible mod $p$ representations occur as irreducible constituents of the various $\overline W$. Let $\tau$ be one of these finitely many irreducible constituents, say of dimension $d$. We claim that only finitely many twists $$ \tau\otimes \overline\chi $$ can occur as an irreducible constituent of a geometric representation $\rho$, where $\overline\chi$ is the mod $p$ reduction of a character of $\pi_1(\operatorname{Spec}k)$. Indeed, since $\rho$ is geometric, every irreducible constituent of $\rho$ is geometric. Therefore, if $\tau\otimes\overline\chi$ occurs in $\rho$, then $\tau\otimes\overline\chi$ is geometric. Hence its determinant $$ \det(\tau\otimes\overline\chi) = \det(\tau)\,\overline\chi^{d} $$ is a geometric character of $G$. By the finiteness of geometric characters proved above, the character $$ \det(\tau)\,\overline\chi^{d} $$ has only finitely many possibilities. Since $\det(\tau)$ is fixed, the character $\overline\chi^{d}$ has only finitely many possibilities. A character of $\pi_1(\operatorname{Spec}k)$ with values in $\overline{\mathbb F}_p^\times$ is determined by the image of a Frobenius element, and for each fixed value of $\overline\chi^{d}$ there are only finitely many possible values of $\overline\chi$. Therefore only finitely many twists $\tau\otimes\overline\chi$ can occur. It follows that there are only finitely many possible irreducible constituents of $\rho$. Since $\rho$ is semisimple of dimension $n$, it is a direct sum of at most $n$ irreducible constituents. Hence there are only finitely many isomorphism classes of such representations $\rho$. This proves the theorem.
	\end{proof}

 \subsection{A remark}
	Theorem~\ref{main3} naturally leads to the following question.
	
	\begin{Question}\label{question:lifting}
		Let $p$ be a prime number, and let $k$ be a finite field of characteristic different
		from $p$. Let $X$ be a normal geometrically connected variety over $k$, let
		$X\hookrightarrow \overline X$ be a compactification, and let
		$
		D\in \Div^+_{\overline X\setminus X}(\overline X).
		$
		Is every continuous irreducible geometric representation
		$$
		\rho:\pi_1(X,D)\longrightarrow \GL_n(\overline{\mathbb{F}}_{p}),
		$$
	 liftable to characteristic zero?
	\end{Question}
	
	For curves, the answer is positive; this is covered by a conjecture of de Jong and its solution established by the work of Böckle--Khare and Gaitsgory, see \cite{afinite}. In higher dimension, however, such a lifting statement is false in general. We briefly
	explain a source of counterexamples. Let $q$ be a power of $p$ with $q\geq 7$,
	and let $
	G=\mathrm{SL}_2(\mathbb F_q)$
	be the two dimensional special linear group over the finite field of order $q$. The natural representation
	$$
	\iota:G\hookrightarrow \GL_2(\overline{\mathbb F}_p)
	$$
	does not lift to a two-dimensional representation of $G$ over a field of
	characteristic zero. Indeed, if such a lift existed, then $G$ would embed into
	$\GL_2(\mathbb C)$. Then its quotient $\mathrm{PSL}_2(\mathbb F_q)$ would be
	a finite subgroup of $\mathrm{PGL}_2(\mathbb C)$. But the finite subgroups of
	$\mathrm{PGL}_2(\mathbb C)$ are cyclic, dihedral, $A_4$, $S_4$, and $A_5$, whereas
	$\mathrm{PSL}_2(\mathbb F_q)$ is not of this form for $q\geq 7$.
	
	Now take the split extension
	$$
	1\longrightarrow G\longrightarrow G\times \Gal(\overline k/k)
	\longrightarrow \Gal(\overline k/k)\longrightarrow 1.
	$$
	By \cite[Theorem A]{MR3806728}, for every integer $r\geq 2$ there exists a smooth projective
	geometrically connected variety $X$ over $k$ of dimension $r$ such that
	$$
	\pi_1(X)\simeq G\times \Gal(\overline k/k).
	$$
	Under this identification, the geometric fundamental group maps onto $G$. Therefore
	the representation
	$$
	\pi_1(X)\longrightarrow G\stackrel{\iota}{\longrightarrow}
	\GL_2(\overline{\mathbb F}_p)
	$$
	is continuous, irreducible, and geometric. If it admitted a lift to characteristic zero,
	then its restriction to the geometric fundamental group would give a characteristic-zero
	lift of the natural two-dimensional representation of $G$, contradicting the preceding
	paragraph. 
	
	Thus, in dimension at least two, not every irreducible geometric mod $p$
	representation can be lifted to characteristic zero. This shows that the
	liftability hypothesis in Theorem~\ref{main3} cannot simply be removed by a
	general lifting theorem in higher dimension. The example above comes from
	the projective case, where realization theorems for fundamental groups allow
	one to force prescribed finite quotients of the geometric fundamental group.
	Therefore, if one wants to seek a positive higher-dimensional lifting
	statement, one should at least restrict the question to affine varieties.	On the other hand, for the purpose of proving Conjecture~\ref{conj:main},
	one may reduce to the case where $X$ is affine. Thus, if one hopes for a positive lifting statement relevant to the present finiteness problem, it should formulate it in the affine case rather than in the
	projective case.

	\section*{Acknowledgements}
	
	The author would like to thank the anonymous referee for carefully reading the manuscript and for valuable comments and suggestions, which improve the quality of this paper.

	\bibliographystyle{plain}
	\bibliography{varity}

\end{document}